\documentclass[a4paper,12pt,english]{article}
\usepackage[cp1251]{inputenc}
\usepackage[T2A]{fontenc}
\usepackage[english]{babel}
\usepackage[dvips]{graphicx}
\usepackage{amsmath}
\usepackage{amssymb}
\usepackage{amsxtra}
\usepackage{amsthm,amscd}
\usepackage{tikz-cd}

\usepackage{geometry}
\newcommand{\oq}{\omega}
\newcommand{\aq}{\alpha}
\newcommand{\bq}{\beta}
\newcommand{\cq}{\gamma}

\def\dl{\delta}
\def\vt{\vartheta}
\def\vk{\varkappa}
\def\sg{\sigma}
\def\mL{\mathfrak{L}}
\def\mE{\mathcal{E}}
\def\bi{\mathbf{i}}
\def\be{\mathbf{e}}
\newcommand{\el}[1]{[#1]}
\def\mB{\mathfrak{B}}
\def\ou{\overline{u}}
\def\ma{\mathfrak{a}}

\newcommand{\ee}{u}
\newcommand{\prt}{\partial}

\newcommand{\bb}[1]{\langle #1 \rangle}

\newcommand {\bR}{\mathbf{R}}
\def\mA{\mathfrak{A}}
\def\mX{\mathfrak{X}}
\def\mf{\mathfrak{f}}
\def\mR{\mathfrak{R}}
\def\mG{\mathfrak{g}}
\def\bF{\mathbf{F}}

\def\Tor{\mathbf{T}}
\def\vp{\varphi}
\def\ve{\varepsilon}
\def\mod{\mathop{\rm mod}\nolimits}

\newcommand{\id}[1]{\mathop{\rm id}\nolimits_{#1}}

\newcommand{\ls}{\leqslant}

\newtheorem{theorem}{Theorem}

\begin{document}

\title{Reduction in mechanical systems with symmetry\footnote{Submitted on January 9, 1974.}}

\author{M.P.\,Kharlamov\footnote{Moscow State University.}}

\date{}

\maketitle

\begin{center}
{\bf \textit{Mekh. Tverd. Tela} (Russian Journal ``Mechanics of Rigid Body''),\\
1976, No. 8, pp. 4--18}

\vspace{5mm}


\vspace{2mm}


\end{center}

\vspace{3mm}

\section*{Introduction}

The first part of the article is, in fact, the classical Routh method delivered in the language of contemporary theory of Lagrangian systems. But the Routh method deals only with concrete equations and, therefore, can be applied only in the case when the configuration spaces of the initial and the reduced systems are manifolds diffeomorphic to open domains in the Euclidean space. The approach described below gives a possibility to find the structure of these manifolds in the general case and also to reveal some properties of the reduced system, first of all, the existence for this system of a global Lagrange function. The notion of a mechanical system used here was introduced in \cite{Smale}, where the necessary properties of such systems can be found.

As an example of the application of the described method we present the global reduction in the problem of the motion of a rigid body having a fixed point in a potential force field with a symmetry axis; this axis is fixed in the inertial space and drawn through the fixed point of the body. We present the complete proof of the theorem formulated by G.V.\,Kolosov \cite{Kolosov} on the equivalence of the reduced system in this case to the problem of the motion of a material point over an ellipsoid and also some corollaries of this theorem based on the results of \cite{Lust}.

\section{General theory}
We consider a mechanical system with symmetry $(M,K,V_0,G)$, where $M$ is a manifold (the configuration space of the system), $K$ a Riemann metric on $M$, $V_0$ a function on $M$, and $G$ a Lie group acting on $M$ and preserving $V_0$. All objects are supposed to be  $C^\infty$-smooth. The action of $G$ on $M$ is extended to $TM$ with tangent maps. The resulting group of diffeomorphisms is denoted by $G_T=\{Tg: g\in G\}$. The metric $K$ is supposed invariant under the action of $G_T$.

In what follows we deal only with the case of commutative $G$ isomorphic to $\bR^k{\times}\Tor^\ell$ with the natural Lie group structure. In addition we suppose that there exists a principle bundle $(M,G,S)$, which means that (see e.g. \cite{Bishop})

1) the action of $G$ is free (the only element having fixed points is the unit);

2) the manifold $S$ is a factor manifold of $M$ with respect to the action of $G$, the projection $\pi: M \to S$ is $C^\infty$-smooth;

3) $M$ is locally trivial, i.e., for each point $s\in S$ there exist a neighborhood $U$ of $s$ and a $C^\infty$-map $F_U:\pi^{-1}(U) \to G$ such that $F_U$ commutes with any $g\in G$ (supposing $G$ acts on itself by means of translations) and the map $\pi^{-1}(U) \to U{\times}G$ defined as $m\mapsto (\pi(m),F_U(m))$ is a diffeomorphism.

Let us define a \textbf{standard chart} on $S$ as a chart $(U,\vp)$ satisfying the following conditions:

(i) $\vp$ is a homeomorphism of $U$ onto the open disk $D^n$ of the Euclidean space;

(ii) for the domain $U$ there exists the above described map $F_U$.

All charts considered below will be supposed standard without loss of generality. The domain of a standard chart will be called a standard subset of the manifold $S$. Our goal is to construct on $TS$ such a dynamical system the trajectories of which together with initial points in $M$ uniquely define the corresponding trajectories of the initial system.

As usual, for an arbitrary manifold $\mathfrak{M}$ we call the charts on $T\mathfrak{M}$ and $T^*\mathfrak{M}$ of the type $(TW,T\theta)$ and $(T^*W,T^*\theta)$ obtained from some chart $(W,\theta)$ on $\mathfrak{M}$ the \textbf{natural charts} and the coordinates in a natural chart the \textbf{natural coordinates}.

Let us introduce coordinates on $M$ of the special type. First, we identify $G$ with $\bR^k{\times}\Tor^\ell$; the elements of $\bR^k{\times}\Tor^\ell$ are defined by $(x,\psi)$, where $x=(x^1,\ldots,x^k)$, $\psi=(\psi^1 \mod 2\pi, \ldots, \psi^\ell\mod 2\pi)$, and the group operation is the sum in each coordinate. Let $(U,\vp)$ be a chart on $S$. Define $u: \pi^{-1}(U) \to D^n{\times}\bR^k{\times}\Tor^\ell$ as $u(m)=(\vp(\pi(m)),F_U(m))$. Thus, for each domain $U$ of a chart on $S$ we obtain the coordinates $u(m)=(q,x,\psi)$ on the open subset $\pi^{-1}(U) \subset M$; here $q=(q^1,\ldots,q^n)\in D^n$. These coordinates will be called \textbf{special}. Let us denote the corresponding natural coordinates on $TM$ and $T^*M$ by $(q,x,\psi,\dot q, \dot x, \dot \psi)$ and $(q,x,\psi, p,y,\zeta)$ respectively. Here, the same as above, for the sake of brevity we put $p=(p^1,\ldots,p^n)$, $y=(y^1,\ldots,y^k)$, $\zeta=(\zeta^1,\ldots,\zeta^\ell)$.

For all $g\in G$ the following commutative diagrams hold
$$
\begin{tikzcd}
M \arrow{rr}{g}\arrow{rd}{\pi} & & M \arrow{ld}{\pi}\\
& S &
\end{tikzcd} \qquad
\begin{tikzcd}
TM \arrow{rr}{Tg}\arrow{rd}{T\pi} & & TM \arrow{ld}{T\pi}\\
& TS &
\end{tikzcd}
$$

The first integral of the moment of the quantity of motion corresponds to the symmetry group $G$. It is called the momentum integral \cite{Smale} $J:TM\to \mG^*$. Here $\mG$ is the Lie algebra of $G$ and $\mG^*$ is the dual space to $\mG$; both $\mG$ and $\mG^*$ are $(k+\ell)$-dimensional vector spaces. Since $G$ is commutative, its adjoint action on $\mG^*$ dual to the adjoint action on $\mG$ is trivial; the stationary subgroup of any point $\mf \in \mG^*$ coincides with the whole $G$. Then $J_\mf=J^{-1}(\mf)$ is invariant under the action of $G_T$ (see \cite{Smale}, Corollary 4.5).

Let us consider only the trajectories of the initial system with a fixed value $\mf$ of the integral $J$. Putting $\rho=T\pi|{J_\mf}$ we have the commutative diagram for all $g\in G$
$$
\begin{tikzcd}
J_\mf \arrow{rr}{Tg}\arrow{rd}{\rho} & & J_\mf \arrow{ld}{\rho}\\
& TS &
\end{tikzcd}
$$

Let us describe $J$ in special coordinates. By definition, $J(m,v)=\aq_m^*(K^*(m,v))$, where $K^*:TM\to T^*M$ is the bundle isomorphism defined by the Riemann metric $K$, $K^*(m,v_1)(m,v_2)=K_m(v_1,v_2)$, and the map $\aq_m: \mG \to T_m M$ assigns to each $X\in \mG$ the representative at the point $m$ of the vector field generated by the one-parameter subgroup of $G$ corresponding to $X$. The map $\aq_m^*: T_m^* M\to \mG^*$ is dual to $\aq_m$.

Let us fix a special coordinate system on $M$. Suppose that $X_1,\ldots,X_k,\Psi_1,\ldots,\Psi_\ell$ is the basis in $\mG$ such that each vector $\ee ^i X_i+\mu^j \Psi_j$ generates a one-parameter subgroup $g(t)$ of $G$ acting on $M$ in the chosen coordinate system as $g(t)(q,x,\psi)=(q,x+\ee  t, \psi+\mu t)$, where $\ee =(\ee ^1,\ldots,\ee ^k)$, $\mu=(\mu^1,\ldots,\mu^\ell)$. Then for $m=(q,x,\psi)$,
$$
\aq_m(\ee,\mu)=\ee ^i \frac{\prt}{\prt x^i}+\mu^j \frac{\prt}{\prt \psi^j}, \qquad \aq_m^*(p,y,\zeta)=y_id\ee ^i+\zeta_j d \mu^j.
$$
The isomorphism $K^*|{T_m M}$ is given by the matrix $||K_{ij}||$ ($i,j=1,\ldots,n+k+\ell$) of the quadratic form $K_m(v,v)$. Obviously, $K^*(q,x,\psi,\dot q,\dot x,\dot \psi)=(q,x,\psi,K_{\dot q},K_{\dot x},K_{\dot \psi})$, where, for instance, $K_{\dot q}=(K_{\dot q^1},\ldots, K_{\dot q^n})$, $K_{\dot q^\bq}$ is the partial derivative with respect to ${\dot q}^\bq$ of the function $K\circ Tu^{-1}$ at the point $(q,x,\psi,\dot q,\dot x,\dot \psi)$ and the function $K:TM\to \bR$ is defined as $K(m,v)=\frac{1}{2}K_m(v,v)$. Finally, the map $J$ in coordinate form is
\begin{equation}\label{eq01}
  J(q,x,\psi,\dot q,\dot x,\dot \psi)=K_{\dot x^i}d\ee ^i+K_{\dot \psi^j}d\mu^j.
\end{equation}

\vspace{2mm}
\textbf{Remark 1.} Denote by $D$ the matrix of the order $k+\ell$ which is the right lower block of the matrix $||K_{ij}||$, i.e., $D=||K_{ij}||$ for $i,j=n+1,\ldots,n+k+\ell$. Since the quadratic form $K_m(v,v)$ is positively definite, $\det D \ne 0$.
\vspace{2mm}

\textbf{Proposition 2.} The pre-image of any point in $TS$ under the map $\rho=T\pi|{J_\mf}$ consists exactly of one orbit of the group $G_T$.
\vspace{2mm}

\begin{proof}
In special coordinates $Tg(q,x,\psi,\dot q,\dot x,\dot \psi)=(q,x+x_g,\psi+\psi_g)$,
where $x_g\in \bR^k$, $\psi_g=(\psi_g^1,\ldots,\psi_g^\ell)$, $0\ls \psi_g^j < 2\pi$, $j=1,\ldots,\ell$. In the natural chart on $TS$ corresponding to the coordinates $(q^1,\ldots,q^n)$ we write $T\pi(q,x,\psi,\dot q,\dot x,\dot \psi)=(q,\dot q)$. Fix a point  $\tau=(q,\dot q)\in TS$. The manifold $(T\pi)^{-1}(\tau)$ is diffeomorphic to $G{\times}\mG$. Moreover, if we choose $(x,\psi,\dot x,\dot \psi)$ as coordinates in $(T\pi)^{-1}(\tau)$, then any constant section $G{\times}\{(\dot x_0,\dot \psi_0)\}$ appears to be an orbit of $G_T$. Since $J$ is equivariant and $G$ trivially acts on $\mG^*$, $J$ is constant on the orbits of $G_T$. Therefore the restriction $J^{(\tau)}: G{\times}\mG\to \mG^*$ of $J$ to the pre-image of $\tau$, linear on each fiber $\{g\}{\times}\mG$, is constant on the sections of the form $G{\times}\{(\dot x_0,\dot \psi_0)\}$.
Obviously, $\rho^{-1}(\tau)=(J^{(\tau)})^{-1}(\mf)$ and it is exactly one orbit of $G_T$ if and only if $J^{(\tau)}|{\{g\}{\times}\mG}:\mG\to \mG^*$ is an isomorphism. This is equivalent to the condition that the system
\begin{equation}\label{eq02}
\begin{array}{ll}
    K_{\dot x ^i}(q,\dot q,\dot x,\dot \psi)=\xi_i & (i=1,\ldots,k),\\
    K_{\dot \psi ^j}(q,\dot q,\dot x,\dot \psi)=\eta_j & (j=1,\ldots,\ell)
\end{array}
\end{equation}
with fixes $q$ and $\dot q$ has a unique solution with respect to $(\dot x,\dot \psi)$. Here $\mf=\xi_id\ee ^i+\eta_jd \mu^j$ and the function $K$ is independent of $x,\psi$ due to $G_T$-invariance. But this system is linear with respect to $\dot x,\dot \psi$ with the non-degenerate matrix $D$. The statement is proved.
\end{proof}

Let us formulate some results from \cite{Abra} (the corresponding numbers from \cite{Abra} are given in parentheses\footnote{During this translation we also added the numbers according to the revised, enlarged, and reset edition of the book by R.Abraham and J.E.Marsden ``Foundations of Mechanics'', Benjamin, Readings, Mass., 1978, 806~p. Some notation was also changed (see Preface to the Second Edition therein).}). We forget for a while about the above notation.

Let $M$ be an $r$-dimensional manifold and $(U,\vp)$ a chart on $M$, $\vp(m)=(q^1,\ldots,q^n)$. Denote by $(q^1,\ldots,q^r, \dot q^1,\ldots,\dot q^r)$ and $(q^1,\ldots,q^r,p_1,\ldots,p_r)$ the corresponding natural coordinates on $TM$ and $T^*M$ respectively.

\parskip=2mm

\textbf{Proposition 3 (14.14, 3.2.10).} Let $M$ be an $r$-manifold and $V=T^*M$. Consider the natural projection $\tau_M^*: V \to M$ and $T\tau_M^*: TV \to TM$. Let $v_m$ $(m\in M)$ denote a point of $V$ and $w_{v_m}$ a point of $TV$ in the fiber over $v_m$. Define
$\theta_{v_m}: T_{v_m}V \to \bR$ as $w_{v_m} \mapsto (v_m \circ T \tau_M^*)(w_{v_m})$ and $\theta_0: v_m \mapsto \theta_{v_m}$. Then $\theta_0$ is a 1-form on $V$, and $\omega_0=-d\theta_0$ is a symplectic form on $V$; $\theta_0$ and $\omega_0$ are called the \textbf{canonical forms} on $V$.

\textbf{Remark 4.} In the natural coordinates $\theta_0=p_i dq^i$ and $\omega_0=dq_i \wedge dp^i$.

\textbf{Definition 5 (17.2, 3.5.2).} Let $M$ be a manifold and let $L:TM \to \bR$ be a smooth function. Then the map
$$
\bF L: TM \to T^*M: v_m \mapsto T_{v_m} L_m \in {\rm Lin}\,(T_m M,\bR)=T_m^*M
$$
is called the \textbf{fiber derivative} of $L$. Here $L_m$ denotes the restriction of $L$ to the fiber $T_m M$ over $m$.

\textbf{Definition 6 (17.7, 3.5.8).} A smooth function $L:TM\to \bR$ is called a \textbf{regular Lagrangian} if $\bF L$ is regular, i.e., the tangent map to $\bF L$ at each point is surjective.

Let $f:M\to N$ be a smooth map of manifolds. Denote by $f^*: \Omega_k(N) \to \Omega_k(M)$ the corresponding map of the spaces of differential $k$-forms.

\textbf{Remark 7.} In the natural coordinates
$$
\displaystyle \bF L(q^1,\ldots,q^r, \dot q^1,\ldots,\dot q^r)=(q^1,\ldots,q^r, L_{\dot q^1},\ldots,L_{\dot q^r}), \qquad L_{\dot q^i} =\frac{\prt}{\prt \dot q ^i}(L\circ T \vp^{-1}).
$$
If $\theta_L=(\bF L)^* \theta_0$, $\omega_L=(\bF L)^* \omega_0$, then in these coordinates $\theta_L=L_{\dot q^i} dq^i$, $\omega_L=dq^i\wedge dL_{\dot q^i}$.

\textbf{Proposition 8 (17.8, 3.5.9).} The function $L$ is a regular Lagrangian if and only if $\omega_L$ is a symplectic form on $TM$.

\textbf{Definition 9 (17.13, 3.5.12).} A \textbf{second-order equation} on a manifold $M$ is a vector field $X$ on $TM$ such that $T\tau_M \circ X$ is the identity on $TM$.

\textbf{Definition 10 (17.15, 3.5.14).} If $c: I\to TM$ ($I=[-\ve,\ve]$) is an integral curve of a vector field $X$ on $TM$, then $\tau_M \circ c: I \to M$ is called a \textbf{base integral curve} of $X$.

\textbf{Proposition 11.} A vector field $X$ on $TM$ is a second-order equation if and only if for any integral curve $c(t)$ of the field $X$ we have $c(t)=T(\tau_M \circ c)(t,1)$, i.e., any integral curve of $X$ equals the derivative of its base integral curve.

\textbf{Proposition 12 (17.16, 3.5.15).} Let $X$ be a vector field on $TM$ and $(U,\vp)$ be a chart on $M$ with $\vp(U)=U' \subset \bR^r$. Suppose that in natural coordinates $X$ has the form
$$
X:U'\times \bR^r \to U'\times \bR^r\times \bR^r\times \bR^r:(q,\dot q)\mapsto (q,\dot q,X_1(q,\dot q),X_2(q,\dot q)).
$$
Then $X$ is a second-order equation if and only if, for every chart, $X_1(q,\dot q)=\dot q$ for all $\dot q\in \bR^r$.

\textbf{Definition 13 (17.18, 3.5.11).} Given a regular Lagrangian $L:TM\to\bR$, define the \textbf{action} $A:TM\to \bR$ of $L$ by $A(v_m)=\bF L(v_m){\cdot}v_m$ and the \textbf{energy} $E$ of $L$ by $E=A-L$. Let $X_E$ be the vector field on $TM$ such that for any vector field $Y$ on $TM$ we have $dE(Y)=\omega_L(X_E,Y)$; $X_E$ exists and is uniquely defined due to the non-degeneracy of $\omega_L$. The dynamical system corresponding to $X_E$ is called the \textbf{Lagrangian system} with the Lagrangian $L$.

The latter term is legitimate due to the following statement.

\textbf{Proposition 14 (17.19-17.20, 3.5.17).} For a regular Lagrangian $L$ the field $X_E$ is a second-order equation and a curve $c:I \to M$ is a base integral curve of $X_E$ if and only if in natural coordinates it satisfies Lagrange's equations
\begin{equation}\label{eq03}
 \frac{d}{dt}L_{\dot q^i}(c(t),c'(t))-L_{q^i} (c(t),c'(t))=0.
\end{equation}
Here, of course, $(c(t),c'(t))=Tc(t,1)$, where $Tc:I{\times}\bR \to TM$.

\textbf{Definition 15 (18.1, 3.6.1).} A smooth function $L:TM\to \bR$ is called a \textbf{hyperregular Lagrangian} if $\bF L: TM \to T^*M$ is a diffeomorphism. In this case $\bF L$ is called the \textbf{Legendre transformation}.

\textbf{Proposition 16 (18.14, 3.6.4).} Let $L$ be a hyperregular Lagrangian. Then its action is $A=\theta_L(X_E)$.

\textbf{Remark 17.} If $L$ is a hyperregular Lagrangian, then its action in the natural coordinates has the form $A(q,\dot q)=\dot q^i L_{\dot q^i}$. It follows immediately from the fact that $X_E$ is a second-order equation.

\vspace{4mm}
\parskip=0mm

We now return to the problem considered. Again $M$ is the space of a mechanical system with symmetry $(M,K,V_0,G)$ and the Lagrangian of this system is $L=K-V$, where $V=V_0\circ \tau_M$. In this case $\bF L=K^*$, therefore, $L$ is a hyperregular Lagrangian. Denote its energy by $E$, then its action is $A=\theta_L(X_E)$. System \eqref{eq02} describing $J_\mf$ in special coordinates is solvable in $\dot x,\dot \psi$ due to Remark 1. Let the solved system be
\begin{equation}\label{eq04}
  \dot x^i=f^i(q,\dot q,\xi,\eta), \quad \dot \psi^j=h^j(q,\dot q,\xi,\eta) \qquad (i=1,\ldots,k; \; j=1,\ldots,\ell).
\end{equation}
In particular, this means that $J_\mf$ is a submanifold in $TM$ of co-dimension $k+\ell$. At the points $v_m\in J_\mf$, the tangent space $T_{v_m}J_\mf$ is given in $T_{v_m}TM$ by the system
\begin{equation}\label{eq05}
  d L_{\dot x^i}=0, \qquad d L_{\dot \psi^j}=0\qquad (i=1,\ldots,k; \; j=1,\ldots,\ell).
\end{equation}
Indeed, since $V$ does not depend on ${\dot x^i},{\dot \psi^j}$ the correspondent partial derivatives of $K$ and $L$ coincide.

Introduce the following notation. Let $F:J_\mf \to \bR$ be invariant under the action of $G$. By Proposition~2, there exists a unique function on $TS$ closing the diagram
$$
\begin{tikzcd}
\bR  & \\
J_\mf \arrow{u}{F}\arrow{r}{\rho}& TS
\end{tikzcd}
$$
Let us denote this function by $\bb{F}$: $F=\bb{F}\circ \rho$. If $F$ is a function on $TM$ preserved by the action of $G_T$, we denote $\bb{F}=\bb{F|{J_\mf}}$. The derivatives of such functions have the form
\begin{equation}\label{eq06}
  \begin{array}{c}
\displaystyle    \bb{F}_{q^\bq}=\bb{{F}_{q^\bq}+{F}_{\dot x^i}\frac{\prt f^i}{\prt q^\bq}+ {F}_{\dot \psi^j}\frac{\prt h^j}{\prt q^\bq}},\qquad
    \bb{F}_{\dot q^\bq}=\bb{{F}_{\dot q^\bq}+{F}_{\dot x^i}\frac{\prt f^i}{\prt \dot q^\bq}+ {F}_{\dot \psi^j}\frac{\prt h^j}{\prt \dot q^\bq}}\\[3mm]
    (\bq=1,\ldots,n).
  \end{array}
\end{equation}

Let $\dl$ stand for the external derivative on $TS$. Let us fix a chart $(U,\vp)$ on $S$, $\vp(s)=(q^1,\ldots,q^n)$. Let $(q,x,\psi)$ be the corresponding special coordinates on $M$. On $TU$, we define the following objects
$$
\vt_\mL(w_s)=\bb{L_{\dot q^\bq}}(w_s)\dl q^\bq, \qquad \Omega_\mL(w_s)= \dl q^\bq \wedge \dl \bb{L_{\dot q^\bq}}(w_s),
$$
where $w_s\in TU$ and $L_{\dot q^\bq}$ is calculated in the coordinates $(q,x,\psi,\dot q,\dot x,\dot \psi)$.

\textbf{Proposition 18.} If $\mf=0$, $\vt_\mL$ does not depend on a chart and defines a 1-form on $TS$. For arbitrary $\mf$, $\Omega_\mL$ is a symplectic 2-form on $TS$.

\vspace{2mm}
\begin{proof} Suppose we have two charts on $S$ with the coordinate transformation $v^\cq=v^\cq(q)$ $(\cq=1,\ldots,n)$. The corresponding transformation of the special coordinates on $M$ is
\begin{equation}\label{eq07}
  \begin{array}{c}
     v^\cq=v^\cq(q),\quad y^i=x^i+\chi^i(q), \quad \zeta^j=\psi^j+\vk^j(q)\quad (i=1,\ldots,k; \; j=1,\ldots,\ell).
  \end{array}
\end{equation}
This transformation does not change the above introduced coordinates on $\mG^*$ and from \eqref{eq02} we get
$$
\bb{L_{\dot q^\bq}}=\bb{
L_{\dot v^\cq}\frac{\prt v^\cq}{\prt q^\bq}+ L_{\dot y^i}\frac{\prt \chi^i}{\prt q^\bq}+ L_{\dot \zeta^j}\frac{\prt \vk^j}{\prt q^\bq}
}=\bb{L_{\dot v^\cq}}\frac{\prt v^\cq}{\prt q^\bq}+ \xi_i\frac{\prt \chi^i}{\prt q^\bq}+ \eta_j\frac{\prt \vk^j}{\prt q^\bq}.
$$
Therefore in the new chart,
\begin{equation}\label{eq08}
\vt_\mL=\bb{L_{\dot v^\cq}}\dl v^\cq+\xi_i\frac{\prt \chi^i}{\prt q^\bq}\dl q^\bq + \eta_j\frac{\prt \vk^j}{\prt q^\bq}\dl q^\bq.
\end{equation}
The condition $\mf=0$ is equivalent to $\xi_i=0,\eta_j=0$ $(i=1,\ldots,k; \; j=1,\ldots,\ell)$, and the first statement follows from \eqref{eq08}. Let us make the change of coordinates in $\Omega_\mL$ recalling that $\dl{\circ}\dl\equiv 0$:
$$
\begin{array}{rl}
\displaystyle   \Omega_\mL &  \displaystyle = \dl q^\bq \wedge \dl \bb{L_{\dot q^\bq}}= \dl q^\bq \wedge \left(
\frac{\prt v^\cq}{\prt q^\bq} \dl \bb{L_{\dot v^\cq}} +  \bb{L_{\dot v^\cq}} \dl \frac{\prt v^\cq}{\prt q^\bq}+ \xi_i \dl \frac{\prt \chi^i}{\prt q^\bq}+ \eta_j \dl \frac{\prt \vk^j}{\prt q^\bq} \right)= \\[3mm]
{} & \displaystyle = \dl v^\cq \wedge \dl \bb{L_{\dot v^\cq}}- \bb{L_{\dot v^\cq}} \dl\circ \dl v^\cq -\xi_i \dl\circ \dl \chi^i-\eta_j \dl\circ \dl \vk^j = \dl v^\cq \wedge \dl \bb{L_{\dot v^\cq}}.
\end{array}
$$
Thus, $\Omega_\mL$ does not depend on a chart and is obviously smooth and closed. It is now sufficient to show that $\Omega_\mL$ is non-degenerate. Let us expand $\dl \bb{L_{\dot q^\bq}}$ with the help of \eqref{eq06} and substitute the partial derivatives of $f^i$ and $h^j$ obtained from \eqref{eq05}. Calculating the determinant of $\Omega_\mL$ we get $\det \Omega_\mL =\left({\det ||K_{ij}||}/{\det D}\right)^2 \ne 0$, therefore $\Omega_\mL$ is non-degenerate. Finally, $(TS,\Omega_\mL)$ is a symplectic manifold.
\end{proof}

\vspace{2mm}
\textbf{Lemma 19.} Let $X,Y\in T_{v_m}J_\mf$. Then $\oq_L(X,Y)=\Omega_\mL(T_{v_m} \rho(X),T_{v_m} \rho(Y))$.
\vspace{2mm}

The proof is by direct calculation in special coordinates using \eqref{eq05}, \eqref{eq06} and Remark 7.

We now construct the dynamical system on $TS$ from the vector field $X_E$ on $TM$. It follows from the coordinate form given in Remark 7 that $\omega_L$ is preserved by the group $G_T$, i.e., for all $g\in G_T$, $v\in TM$, $X,Y\in T_v TM$
\begin{equation}\label{eq09}
  \omega_L(w)(X,Y)=\omega_L(g w)(Tg(X),Tg(Y)).
\end{equation}
Note that $E=K+V$ satisfies $E=E\circ g$ for all $g\in G_T$. Hence,
\begin{equation}\label{eq10}
  d E = dE\circ Tg.
\end{equation}
Pick $v\in TM$, $Y\in T_v TM$, $g\in G_T$ and denote $w=g^{-1}v$. Then from \eqref{eq09}, \eqref{eq10} we get
$$
\begin{array}{l}
  \omega_L(v)(Tg\circ X_E(w),Y)=\omega_L(w)(X_E(w),T g^{-1}(Y))=dE\circ Tg^{-1}(Y)=dE(Y)= \\
  \qquad  =\omega_L(v)(X_E(v),Y).
\end{array}
$$
So, $Tg\circ X_E\circ g^{-1} = X_E$. Since $J_\mf$ is an integral manifold of $X_E$, we have the commutative diagram
\begin{equation}\label{eq11}
\begin{tikzcd}
{} & J_\mf \arrow {ld}{\rho}\arrow{dd}{g}\arrow{r}{X_E} &  TJ_\mf \arrow{dd}{Tg}\arrow{rd}{T\rho} &\\
TS & & & TTS\\
& J_\mf\arrow{lu}{\rho}\arrow{r}{X_E}& TJ_\mf\arrow{ru}{T\rho}
\end{tikzcd}
\end{equation}
Then, according to Proposition 2, the vector field $\mX:TS\to TTS$ is well defined by the relation $\mX=T\rho\circ X_E \circ \rho^{-1}$.

\vspace{2mm}
\textbf{Proposition 20.} $\mX$ is a second-order equation and $\mX=X_{\bb{E}}$ in the symplectic structure $\Omega_\mL$.
\vspace{2mm}

\begin{proof}
By definition, $T\tau_S\circ \mX=T(\tau_S\circ \rho)\circ X_E \circ \rho^{-1}$. At the same time, $\tau_S\circ \rho=(\pi\circ \tau_M)|J_\mf$. Hence, $T\tau_S\circ \mX=T\pi \circ T \tau_M \circ X_E \circ \rho^{-1}$. Since $X_E$ is a second-order equation, we have $T \tau_M \circ X_E=\id{TM}$. Therefore $T\tau_S\circ \mX=T\pi \circ \rho^{-1}=\id{TS}$ and $\mX$ is a second-order equation on $S$.

Let us write down the definition of $X_{\bb{E}}$. Let $w_s \in TS$. Then for all $Y\in T_{w_s}TS$
\begin{equation}\label{eq12}
  \Omega_\mL(w_s)(X_{\bb{E}}(w_s),Y)=\dl \bb{E}(w_s)(Y).
\end{equation}
Note that $E|J_\mf=\bb{E}\circ \rho$, therefore,
\begin{equation}\label{eq13}
  dE|J_\mf =\dl \bb{E}\circ T\rho.
\end{equation}
Let $v\in \rho^{-1}(w_s)$. For any $Y\in T_{w_s}TS$ there exists $Y_0\in T_v TM$ such that $T\rho(Y_0)=Y$. From \eqref{eq13} we have
\begin{equation}\label{eq14}
  \dl \bb{E}(Y)=dE(Y_0)=\omega_L(X_E,Y_0).
\end{equation}
By Lemma 19, $\omega_L(X_E,Y_0)=\Omega_\mL(\mX,Y)$. Then \eqref{eq12} and \eqref{eq14} yield $\mX=X_{\bb{E}}$.
\end{proof}

\vspace{2mm}
\textbf{Proposition 21.} Let $a:I\to S$ be a base integral curve of $X_{\bb{E}}$, $b:I\to M$ a base integral curve of $X_E|J_\mf$. If $T(\pi\circ b)(0,1)=Ta(0,1)$, then $\pi\circ b \equiv a$.
\vspace{2mm}

\begin{proof}
According to Propositions 11, 14, and 20, $c(t)=Ta(t,1)$ and $d(t) = Tb(t,1)$ are trajectories of $X_{\bb{E}}$ and $X_E$ respectively. In particular, $Td(t,1)=X_E\circ d(t)$. Let us calculate
$
T(\rho\circ d)(t,1)=T\rho \circ X_E\circ d(t)=X_{\bb{E}}\circ \rho\circ d(t),
$
i.e., $\rho\circ d(t)$ is a trajectory of $X_{\bb{E}}$. It follows from the uniqueness theorem, that if $\rho\circ d(0)=c(0)$, then $\rho\circ d \equiv c$, or $T(\pi\circ b)(t,1)\equiv Ta(t,1)$. Hence, $\pi\circ b \equiv a$.
\end{proof}

\vspace{2mm}
\textbf{Definition 22.} The dynamical system generated by the vector field $X_{\bb{E}}$ is called the \textbf{reduced system} on $TS$.
\vspace{2mm}

\textbf{Remark 23.} Let $(q(t),\dot q(t))$ be a trajectory of the reduced system with initial conditions $q(0)=q_0,\dot q(0)=\dot q_0$. In virtue of \eqref{eq04} and Proposition 21, the base trajectory on $M$ in special coordinates corresponding to the trajectory of $X_E$ with initial conditions $q(0)=q_0,\dot q(0)=\dot q_0, x(0)=x_0, \psi(0)= \psi_0$ has the form
$$
\begin{array}{c}
\displaystyle  q^\bq = q^\bq(t), \quad x^i=x_0^i+\int_0^t f^i(q(t),\dot q(t),\xi,\eta) dt,\quad \psi^j=\psi_0^j+\int_0^t h^j(q(t),\dot q(t),\xi,\eta) dt
 \\[3mm]
  (\bq=1,\ldots,n, \quad i=1,\ldots,k, \quad j=1,\ldots, \ell).
\end{array}
$$
Hence, the trajectories of the whole system are restored from the trajectories of the reduced system by direct integration. Note that for the trajectory of $X_E$ lying in $J_\mf$ the values $\dot x(0), \dot \psi(0)$ cannot be chosen arbitrary, but are found from \eqref{eq04}.
\vspace{2mm}

\section{Properties of the reduced system}
\begin{theorem}\label{theo1} If $\mf=0$ the reduced system is a Lagrangian system with the Lagrangian $\mL=\bb{L}$.
\end{theorem}
\begin{proof}
Let us show that $\bF \mL: TS\to T^*S$ is a bundle isomorphism. Take a chart $(U,\vp)$ on $S$, $\vp: s\mapsto (q^1,\ldots,q^n)$. In the natural coordinates
\begin{equation}\label{eq15}
  \bF\mL: (q^1,\ldots,q^n,\dot q^1,\ldots,\dot q^n)\mapsto (q^1,\ldots,q^n,\mL_{\dot q^1},\ldots,\mL_{\dot q^n}).
\end{equation}
Let $(q,x,\psi)$ be the special coordinates corresponding to $(U,\vp)$. Note that for $\mf=0$,
\begin{equation}\label{eq16}
  \mL_{\dot q^\bq}=\bb{L_{\dot q^\bq}}, \qquad \bq=1,\ldots,n.
\end{equation}
The map $(\dot q,\dot x,\dot \psi) \mapsto (L_{\dot q},L_{\dot x},L_{\dot \psi})$ is an isomorphism $\bR^{n+k+\ell}\to\bR^{n+k+\ell}$ and the map $(\dot x,\dot \psi) \mapsto (L_{\dot x},L_{\dot \psi})$ is one-to-one (see Remark 1). Then $\bF\mL$ is an isomorphism on fibers. Obviously, $\bF\mL$ is smooth. Moreover, due to the positive definiteness of $K$, $T_{w_s}\bF\mL$ has full rank for all $w_s\in TS$. Thus, $\bF\mL$ is a diffeomorphism and $\mL$ is a hyperregular Lagrangian on $TS$.

Obviously, $\bF\mL$ takes $\vt_\mL$ and $\Omega_\mL$ to the canonical forms on $T^*S$ (see \eqref{eq15}, \eqref{eq16} and Remarks 4 and 7). Let $\mE$ be the energy of $\mL$. According to Proposition 16,
\begin{equation}\label{eq17}
  \mL=\vt_\mL(\mE)-\mE.
\end{equation}
In the chart $(U,\vp)$
\begin{equation}\label{eq18}
  \vt(X_{\bb{E}})=\dot q^\bq \mL_{\dot q^\bq},
\end{equation}
since $X_{\bb{E}}$ is a second-order equation (Propositions 12 and 20). Then by Remark 17
\begin{equation}\label{eq19}
  \vt_\mL(X_\mE)=\vt_\mL(X_{\bb{E}}).
\end{equation}
The Lagrangian of the system on $M$ is $L=\theta_L(X_E)-E$. Since $\mf=0$, using Remark 17 and equations \eqref{eq16}, \eqref{eq18} we get
$$
\bb{\theta_L(X_E)}=\bb{\dot q^\bq L_{\dot q^\bq}}=\dot q^\bq \bb{L_{\dot q^\bq}}=\vt_\mL(X_{\bb{E}}).
$$
Hence,
\begin{equation}\label{eq20}
  \mL=\bb{L}=\vt_\mL(X_{\bb{E}})-\bb{E}.
\end{equation}
Comparing \eqref{eq17}, \eqref{eq19}, and \eqref{eq20}, we see that $\mE=\bb{E}$, therefore $X_\mE=X_{\bb{E}}$. \end{proof}

\begin{theorem}\label{theo2}
In the case $\mf \ne 0$ the reduced system is locally Lagrangian in the following sense. Suppose $(U,\vp)$ is a chart on $S$, $\vp: s\mapsto (q^1,\ldots,q^n)$ and $(q,x,\psi)$ the corresponding  special coordinates in $\pi^{-1}(U)$. Having equations \eqref{eq04} valid on $J_\mf\cap \pi^{-1}(U)$, consider the function $\mL:TU\to\bR$ which in the chart $(TU,T\vp)$ has the form
$$
\mL=\bb{L}-\xi_if^i-\eta_j h^j.
$$
The dynamical system $X_{\bb{E}}|TU$ on the manifold $TU$ is a Lagrangian system with the Lagrangian $\mL$. In particular, a curve $c:I\to U$ is a base integral curve of this system if and only if in any chart $(U,\sg)$, $\sg: s \mapsto (v^1,\ldots,v^n)$ it satisfies the equations
\begin{equation}\label{eq21}
 \frac{d}{dt}\mL_{\dot v^\cq}(c(t),c'(t))-\mL_{v^\cq} (c(t),c'(t))=0 \qquad (\cq=1,\ldots,n).
\end{equation}
\end{theorem}

\begin{proof}
Consider a 1-form $\vt_\mL$ on $TU$ given in the chart $(U,\sg)$ by $\vt_\mL(w_s)=\bb{L_{\dot q^\bq}}\dl q^\bq$, $w_s\in TU$. Let $v^\cq=v^\cq(q^1,\ldots,q^n)$ be the transition functions from $(U,\vp)$ to $(U,\sg)$. The corresponding change of the special coordinates in $\pi^{-1}(U)$ is given by \eqref{eq07}. Similar to \eqref{eq08}, in the chart $(U,\sg)$ we have
\begin{equation}\label{eq22}
\vt_\mL=\left\{ \bb{L_{\dot v^\cq}}+\xi_i\frac{\prt \chi^i}{\prt q^\bq} \frac{\prt q^\bq}{\prt v^\cq}  + \eta_j\frac{\prt \vk^j}{\prt q^\bq} \frac{\prt q^\bq}{\prt v^\cq}\right\} \dl v^\cq.
\end{equation}
Let us calculate the values
\begin{equation}\label{eq23}
\mL_{\dot v^\cq}=\bb{L}_{\dot v^\cq}- \xi_i\frac{\prt f^i}{\prt \dot q^\bq} \frac{\prt q^\bq}{\prt v^\cq}  - \eta_j\frac{\prt h^j}{\prt \dot q^\bq} \frac{\prt q^\bq}{\prt v^\cq}.
\end{equation}
According to \eqref{eq06} $\displaystyle \bb{L}_{\dot v^\cq}=\bb{L_{\dot v^\cq}}+ \xi_i\frac{\prt \dot y^i}{\prt \dot v^\cq} + \eta_j\frac{\prt \dot \zeta^j}{\prt \dot v^\cq}$, and from \eqref{eq07}
\begin{equation*}
\frac{\prt \dot y^i}{\prt \dot v^\cq}= \frac{\prt f^i}{\prt \dot q^\bq} \frac{\prt q^\bq}{\prt v^\cq}  + \frac{\prt \chi^i}{\prt q^\bq} \frac{\prt q^\bq}{\prt v^\cq}, \qquad
\frac{\prt \dot \zeta^j}{\prt \dot v^\cq}= \frac{\prt h^j}{\prt \dot q^\bq} \frac{\prt q^\bq}{\prt v^\cq}  + \frac{\prt \vk^j}{\prt q^\bq} \frac{\prt q^\bq}{\prt v^\cq}.
\end{equation*}
Substituting these values in \eqref{eq23}  we obtain
\begin{equation}\label{eq24}
\mL_{\dot v^\cq}=  \bb{L_{\dot v^\cq}}+\xi_i\frac{\prt \chi^i}{\prt q^\bq} \frac{\prt q^\bq}{\prt v^\cq}  + \eta_j\frac{\prt \vk^j}{\prt q^\bq} \frac{\prt q^\bq}{\prt v^\cq}.
\end{equation}
Comparing \eqref{eq22} and \eqref{eq24}, we see that in an arbitrary chart
\begin{equation}\label{eq25}
  \vt_\mL=\mL_{\dot v^\cq} \dl  v^\cq.
\end{equation}
In the same way as in the proof of Theorem~\ref{theo1}, using in particular \eqref{eq16}, we can show that $\mL$ is a hyperregular Lagrangian on $TU$ and, according to \eqref{eq25}, $\bF\mL$ takes $\vt_\mL$ and $\Omega_\mL=-\dl \vt_\mL$ to the canonical forms on $T^*U$. Again
\begin{equation}\label{eq26}
  \mL=\vt_\mL(X_\mE)-\mE,
\end{equation}
where $\mE$ is the energy of $\mL$, and
\begin{equation}\label{eq27}
  \vt_\mL(X_{\bb{E}})=\vt_\mL(X_\mE).
\end{equation}
In this case $\bb{\theta_L(X_E)}=\vt_\mL(X_{\bb{E}})+\xi_i f^i+\eta_j h^j$. Then
\begin{equation}\label{eq28}
  \mL= \bb{L}- \xi_i f^i-\eta_j h^j=\bb{\theta_L(X_E)}-\bb{E}- \xi_i f^i-\eta_j h^j=\vt_\mL(X_{\bb{E}})-\bb{E}.
\end{equation}
From \eqref{eq26} -- \eqref{eq28} we get $\mE=\bb{E}$ and, consequently, $X_\mE=X_{\bb{E}}$, since $\Omega_\mL=-\vt_\mL$ coincides with the restriction to $TU$ of the above defined form $\Omega_\mL$. The second part of the theorem follows from Proposition 14.
\end{proof}

To obtain the next statement, for each standard subset $U\subset S$ let us fix some standard chart $(U,\vp)$. Then for each $U$ by the construction described in Theorem~\ref{theo2} we define a local Lagrangian $\mL_U$ on $TU$. We want to find out when these functions can be ``glued'' to produce one function on $TS$ which is a Lagrangian of the reduced system.

Let $U$ and $V$ be standard sets with non-empty intersection. It follows from \eqref{eq04}, \eqref{eq07} and the definition of the local Lagrangians on $TU$ and $TV$ that there exists a function $\oq_{U,V}$ on $U\cap V$ such that $(\mL_U-\mL_V)|U\cap V = d\oq_{U,V}$. The following properties are obvious:

1) $d\oq_{U,V}=-d\oq_{V,U}$ for all $U,V$;

2) if $U,V,W$ have a non-empty intersection, then $d\oq_{U,V}+d\oq_{V,W}+d\oq_{W,U}=0$ on $U\cap V\cap W$.

Suppose $\mA=\{U_\aq\}_{\aq\in A}$ is a cover of $S$ by standard subsets. Then the formula $c_\mA (\aq_0,\aq_1)=d\oq_{U_{\aq_0},U_{\aq_1}}$ define a 1-dimensional cochain $c_\mA$ of this cover with coefficients in the sheaf $\mR$ of germs of closed 1-forms on $S$ \cite{Hirz}. The above mentioned properties 1 and 2 mean that this cochain is a cocycle and therefore defines an element of the group $H^1(\mA,\mR)$. Denote this element by $\el{c_\mA}$. Let us show that the set of $\el{c_\mA}$ for all standard covers $\mA$ defines an element of the group $H^1(S,\mR)=\lim\limits_{\longrightarrow} H^1(\mA,\mR)$. It is sufficient to show that if a cover $\mB=\{V_\bq\}_{\bq\in B}$ is a refinement of $\mA=\{U_\aq\}_{\aq\in A}$ and $\nu: B \to A$ is a refinement map, i.e., $V_\bq \subset U_{\nu(\bq)}$ for all $\bq\in B$, then the cocycles $\nu^*c_\mA$ and $c_\mB$ are cohomologic. Define a 0-cochain of the cover $\mB$ by $c^0(\bq)=d\oq_{U_{\nu(\bq)},V_\bq}$. Then
$$
\dl ^0 c^0(\bq_0,\bq_1)
=
d\oq_{U_{\nu(\bq_1)},V_{\bq_1}}|(V_{\bq_0}\cap V_{\bq_1})-d\oq_{U_{\nu(\bq_0)},V_{\bq_0}}|(V_{\bq_0}\cap V_{\bq_1}).
$$
This by definition of $\oq_{U,V}$ gives
$$
\dl ^0 c^0(\bq_0,\bq_1)
=
\left(
\mL_{U_{\nu(\bq_1)}}- \mL_{V_{\bq_1}}-\mL_{U_{\nu(\bq_0)}}+ \mL_{V_{\bq_0}}
\right)|(V_{\bq_0}\cap V_{\bq_1}).
$$
Consider the restrictions to $V_{\bq_0}\cap V_{\bq_1}$ of the equalities
$$
\mL_{U_{\nu(\bq_0)}}-\mL_{V_{\bq_0}} =d\oq_{U_{\nu(\bq_0)},V_{\bq_0}}, \qquad \mL_{U_{\nu(\bq_1)}}-\mL_{V_{\bq_1}} =d\oq_{U_{\nu(\bq_1)},V_{\bq_1}}
$$
and subtract the first equality from the second one. Then we get
$$
\begin{array}{l}
  \left(
\mL_{U_{\nu(\bq_1)}}- \mL_{V_{\bq_1}}-\mL_{U_{\nu(\bq_0)}}+ \mL_{V_{\bq_0}}
\right)|(V_{\bq_0}\cap V_{\bq_1})= \\
\qquad  = \left(d\oq_{U_{\nu(\bq_1)},V_{\bq_1}}-d\oq_{U_{\nu(\bq_0)},V_{\bq_0}}\right)|(V_{\bq_0}\cap V_{\bq_1})\stackrel{{\rm def}}{=} (\nu^*c_\mA-c_\mB)(\bq_0,\bq_1).
\end{array}
$$
Finally, we have $\dl ^0 c^0=\nu^*c_\mA-c_\mB$.

Thus, the reduced system uniquely defines an element $c=\lim\limits_{\longrightarrow}\el{c_\mA} \in H^1(S,\mR)$.

\begin{theorem}\label{theo3} The reduced system has a global Lagrangian of the form $\mL=\bb{L}+F+P\circ \tau_S$, where $P$ is a function on $S$ and $F$ is a 1-form on $S$ considered as a function on $TS$ linear on fibers, if and only if the corresponding element $c$ of $H^1(S,\mR)$ equals zero.
\end{theorem}

\begin{proof}
Suppose the reduced system has a global Lagrangian of the needed type. Consider a cover $\mA=\{U_\aq\}_{\aq\in A}$ by standard sets. The local Lagrangians $\mL_{U_\aq}$ and $\mL|TU_\aq$ have the same quadratic part and define the same vector field on $TU_\aq$. Then, as $U_\aq$ is simple connected, there exists $\oq_\aq: U_\aq \to \bR$ such that $\mL=\mL_{U_\aq}+d\oq_\aq$ on $TU_\aq$. If $U_{\aq_0}\cap  U_{\aq_1} \ne \varnothing$, then $d\oq_{U_{\aq_0},U_{\aq_1}}=d\oq_{\aq_1}-d\oq_{\aq_0}$, therefore the cocycle $c_\mA$ is a coboundary, and, consequently, $c=0$.

Now suppose that $c=0$. Then there exists a standard cover $\mA=\{U_\aq\}_{\aq\in A}$ such that $c_\mA$ is a coboundary, i.e., there exists a set of functions $\oq_\aq: U_\aq \to \bR$ such that
$d\oq_{U_{\aq_0},U_{\aq_1}}=d\oq_{\aq_1}-d\oq_{\aq_0}$. Put $\mL=\mL_{U_\aq}+d\oq_\aq$ on $TU_\aq$. Obviously, $\mL$ is a well defined Lagrangian on $TS$ and the corresponding Lagrangian system coincides with the reduced system. This proves the theorem.
\end{proof}

It is known that the group $H^1(S,\mR)$ is isomorphic to the group $H^2(S,\bR)$ of the real cohomologies of the manifold $S$ (see e.g. \cite{Hirz}). From Theorem~\ref{theo3} we have the following sufficient condition for the reduced system to be Lagrangian: if $H^2(S,\bR)$ is trivial, then the reduced system always admits a global Lagrangian.

\section{Application to the rigid body dynamics}
Let us consider the problem of the motion of a rigid body having a fixed point in a force field with a potential $V_0$ invariant under the group of rotations about some axis fixed in space and crossing the fixed point of the body. The problem is described by mechanical system with symmetry; the symmetry group is isomorphic to $S^1$.

It is convenient to represent the configuration space $M$ as the manifold $T^1S^2$, which is the bundle of the unit tangent vectors over the 2-sphere $S^2$. Let us consider this sphere to be the unit sphere in space.  Let $\bi_1,\bi_2,\bi_3$ be an orthonormal frame in space; $\bi_3$ shows the direction of the symmetry axis. Fix the element $z_0=(N,\xi_0)\in T^1 S^2$ such that $N=(0,0,1)$ is the north pole of $S^2$ and $\xi_0$ is the unit tangent vector at $N$ parallel to~$\bi_2$. Let $\be_1,\be_2,\be_3$ be the orths of the principal inertia axes in the body. For each position
$e=(\be_1,\be_2,\be_3)$ of the body there exists a unique element $g_e$ of the group $SO(3)$ which moves $\be_1,\be_2,\be_3$ to $\bi_1,\bi_2,\bi_3$ respectively. Let us assign to the position $e$ the element of $T^1 S^2$ to which $z_0$ is taken by the rotation $g_e \in SO(3)$: $e\mapsto g_e(z_0)$.

\begin{table}[ht]
\def\qq{0.35}
\centering
\begin{tabular}{cc}
\includegraphics[width=\qq\textwidth, keepaspectratio = true]{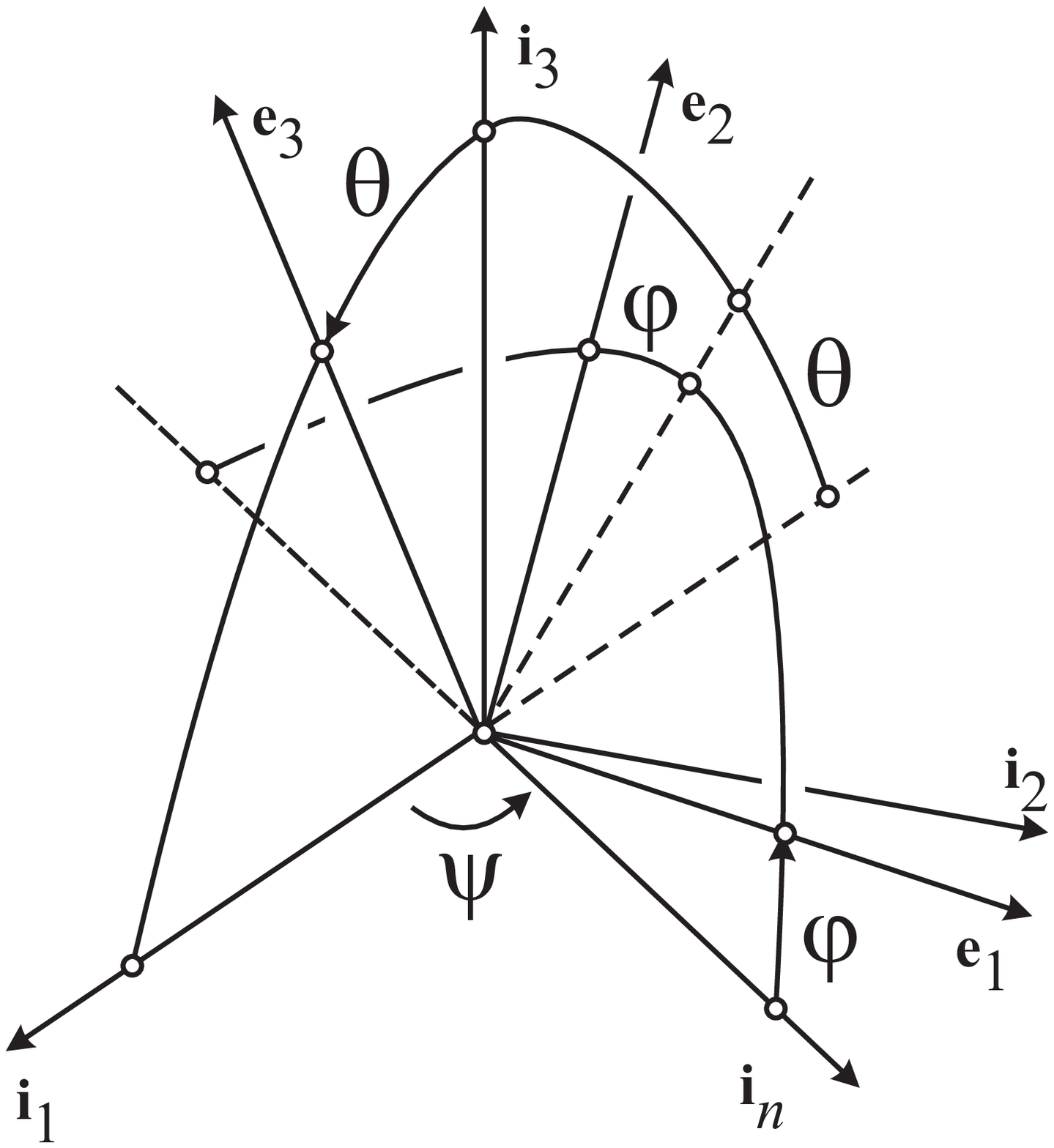} & \includegraphics[width=\qq\textwidth, keepaspectratio = true]{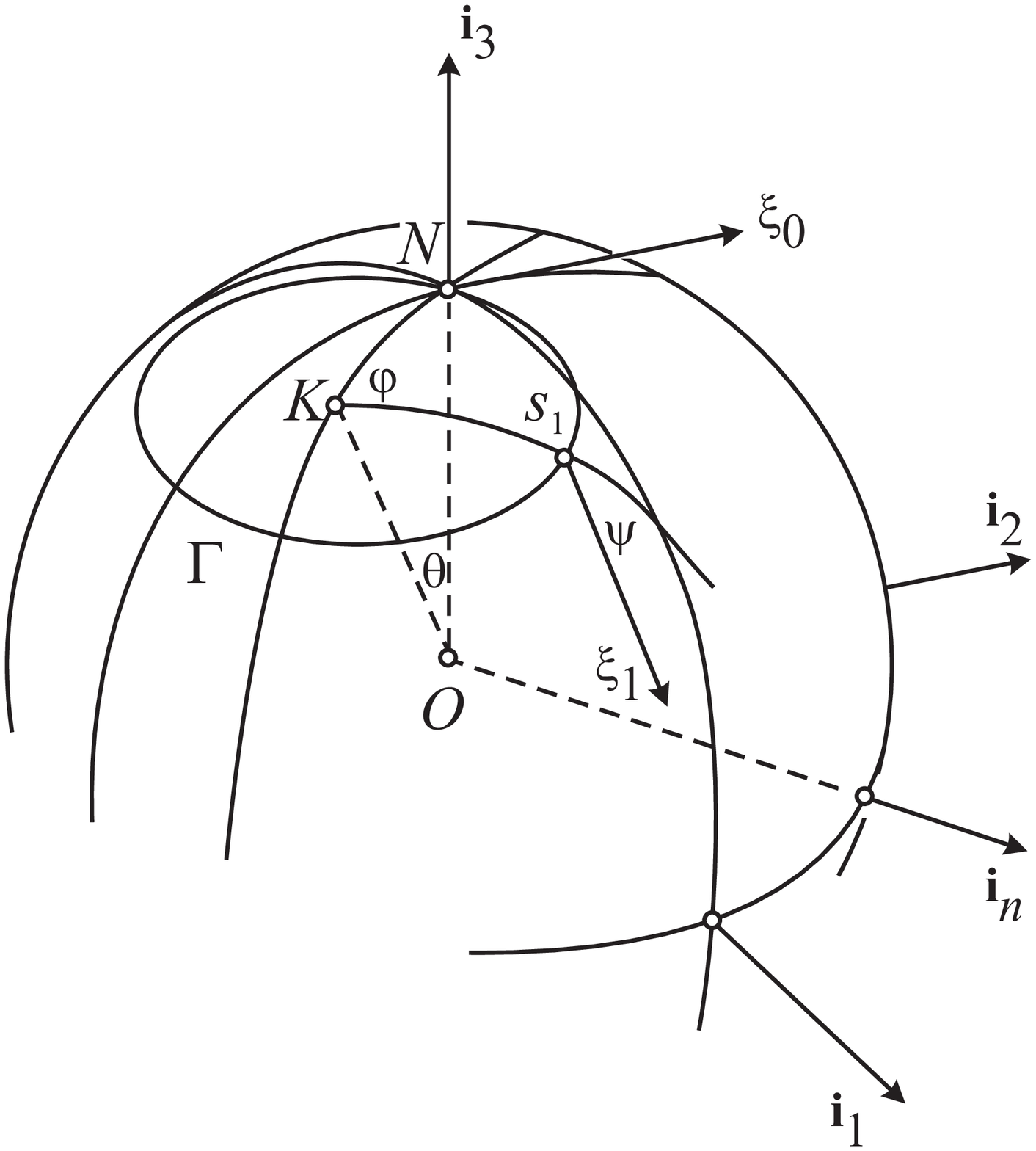}\\
Figure 1 & Figure 2\\
\includegraphics[width=\qq\textwidth, keepaspectratio = true]{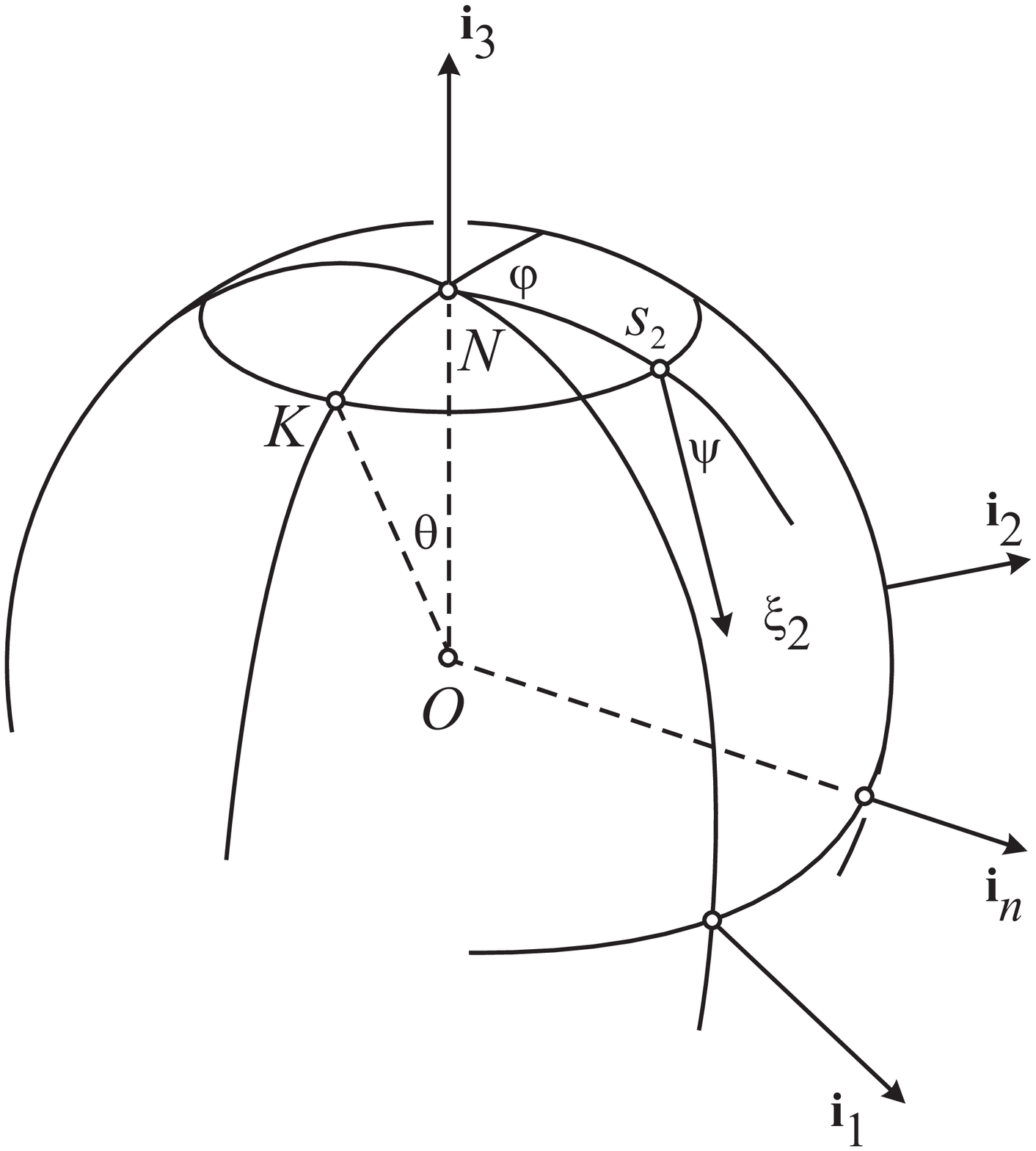} & \includegraphics[width=\qq\textwidth, keepaspectratio = true]{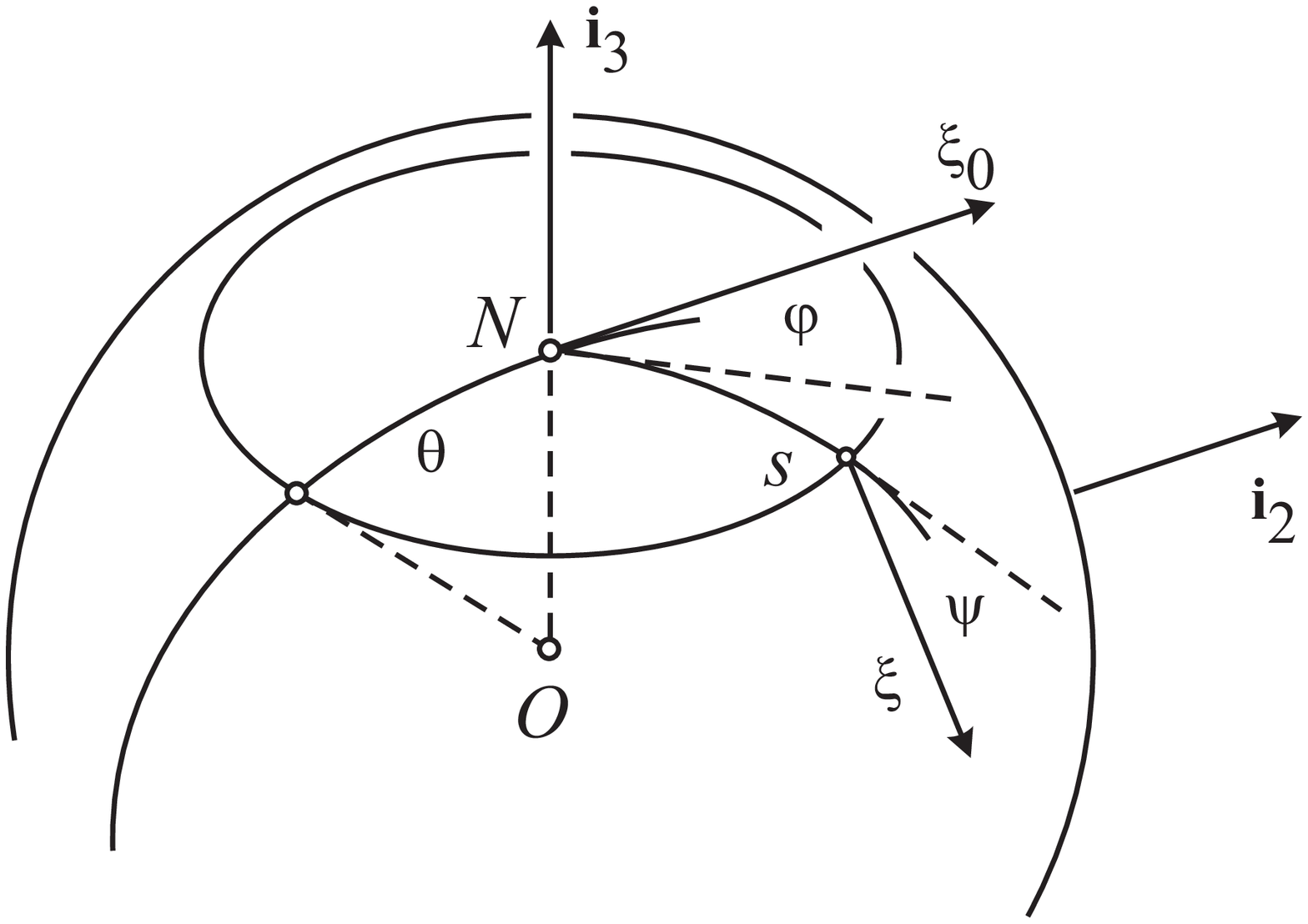}\\
Figure 3 & Figure 4
\end{tabular}
\end{table}

Let us show that, under this identification of $M$ with $T^1S^2$, the angles of proper rotation $\vp$ and nutation $\theta$ correspond to the spherical coordinates on $S^2$. Denote $Q=\{x\in \bR^3:\|x\|\ls \pi\}$. We say that $x\in Q$ is a \textbf{defining vector} of an element $g\in SO(3)$ if $x$ lies on the axis of rotation $g$, $\|x\|$ is the angle of rotation, and the direction of rotation is connected with the direction of $x$ by the right-handed screw rule. The element of $SO(3)$ with the defining vector $x$ will be denoted by $v_x$. If the position $e$ of the body is characterized by the Euler angles $\vp,\psi,\theta$, then the corresponding element $g_e\in SO(3)$ is a composition $g_e=v_{x_3}\circ v_{x_2}\circ v_{x_1}$, where $x_1=-\vp \be_3$, $x_2=-\theta \bi_{n}$, $x_3=-\psi \bi_3$ (see Fig.~1). The sequence of these rotations is shown in Fig.~2\,--\,4. In Fig.~2, the point $K$ is the intersection with the sphere of $O\be_3$, the circle $\Gamma$ is the cross-section of the sphere by the plane orthogonal to $\be_3$ and containing $N$. The element $(s_1,\xi_1)\in T^1 S^2$ is the image of $z_0$ under the rotation $v_{x_1}$.  The rotation $v_{x_2}$ takes $K$ to the north pole $N$ and the circle $\Gamma$ to the cross-section of the sphere orthogonal to $\bi_3$ (see Fig.~3). Finally, $g_e(z_0)=(s,\xi)$ (see Fig.~4) and $\vp,\theta$ are the spherical coordinates of the point $s$. At the same time, the precession angle $\psi$ becomes the angle between the tangent vectors $\xi$ and $\prt/\prt\theta$.

We now investigate the action of the symmetry group on the manifold $T^1S^2$. If the position $e$ is obtained from $e'$ by the rotation $g\in SO(3)$, then $g_{e'}=g_e\circ g$. The equivalence class of $e$ with respect to the action of $G$ is represented in $T^1S^2$ as $g_e\circ G(z_0)$, which is the set of all unit tangent vectors to $S^2$ at the point $g_e(N)\in S^2$. Thus, the map $\pi:M\to S^2$ such that the pre-image of any point of $S^2$ is exactly the equivalence class of the $G$-action is defined without using any coordinates. Namely, $\pi=\tau_{S^2}$ is the projection to the base of the bundle. Obviously, the first two conditions of the existence of a principal bundle hold. To apply the described above procedure of reduction, we have to check the third condition (local triviality).

For any point $s\in S^2$ we can take for $U$ any neighborhood of $s$ not containing poles, and for $F_U$ the map $(\vp,\theta,\psi)\mapsto (\vp,\theta)$, where the Euler angles $(\vp,\theta,\psi)$ (see Fig.~1) are the local coordinates on $M$ except for the pre-images of the poles.

Let us construct the map $F_{U'}$ in some neighborhood $U'$ of the point $N\in S^2$. Let $g_0\in SO(3)$ be an arbitrary element with the only condition that $g_0$ does not take $N$ to itself or to the south pole. Let $U$ be a neighborhood of $g_e(N)$ for which $F_U$ exists. Put $U'=g_0^{-1}(U)$ and $F_{U'}=F_U\circ g_0$. We need to show that $F_{U'}$ commutes with the transformations of the symmetry group $G$. Denote by $z g$ ($z\in T^1S^2,g\in G$) the action of the symmetry group $G$ on the manifold $T^1 S^2$. This notation is used to distinguish it from $g(s), s \in S^2$ and $g(z), z \in T^1 S^2$; the latter means that the element $g\in SO(3)$ is applied to a point $s$ or to a vector $z$ with an origin on $S^2$. Let us show that for all $z\in T^1S^2$, $h\in SO(3)$, $g\in G$
\begin{equation}\label{eq29}
  h(z g)=(h(z))g.
\end{equation}
Indeed, let $g_z$ be an element of $SO(3)$ such that $z=g_z(z_0)$. Then $g_{h(z)}=h\circ g_z$ and $zg=g_z\circ g^{-1}(z_0)$. Hence, $h(zg)=h\circ g_z\circ g^{-1}(z_0)$ and $((h(z))g=g_{h(z)}\circ g^{-1}(z_0)=h\circ g_z\circ g^{-1}(z_0)$. This proves \eqref{eq29}. Now since $F_U$ and, according to \eqref{eq29}, $g_0$ commute with any $g\in G$, this is also true for $F_{U'}$. For the south pole a similar trivialization is built analogously.

Thus, we proved the almost obvious fact that $T^1 S^2$ a total space of a principal bundle with the base $S^2$ and the structure group $G$.

\begin{theorem}[G.V.\,Kolosov]\label{theo4}
If in the problem of the motion of a rigid body in an axially symmetric force field the momentum constant is zero, then the reduced system with the energy constant equal to $h$ is isomorphic to the problem of the motion of a particle over the surface of the ellipsoid $E^2:Ax^2+B y^2+C z^2=1$ $(A,B,C$ are the principle moments of inertia$)$ in the field with the potential
\begin{equation*}
  \frac{ABC (V-h)}{A^2 x^2+B^2 y^2+C^2 z^2}
\end{equation*}
and the energy constant zero. Here $V=\bb{V_0}\circ F^{-1}$ and $F: S^2\to E^2$ is a diffeomorphism. In particular, the motion of a free body $(V_0=0)$ reduces to the geodesic flow on $E^2$ in the metric $d\Sigma=\sqrt{h ABC}(A^2 x^2+B^2 y^2+C^2 z^2)^{-1/2}d\sigma$, where $d\sigma$ is the real metric of the ellipsoid.
\end{theorem}

\begin{proof}
By Theorem~\ref{theo1} the reduced problem is a Lagrangian system on the sphere $S^2$ with the Lagrangian $\mL=\bb{L}$, where $L$ is the Lagrangian of the initial mechanical system. For the local coordinates on the manifold $M$ we take the Euler angles $\vp,\theta,\psi$. Then $\vp$ and $\theta$ are the spherical coordinates on $S^2$. In these coordinates, $\mL$ has the form
$$
\begin{array}{l}
\displaystyle  \mL=\frac{1}{2}\frac{Q \dot \theta^2+R \dot \vp^2-2(A-B)C \dot \vp \dot \theta \sin\vp \cos\vp\sin\theta\cos\theta}{(A\sin^2\vp+B\cos^2\vp)\sin^2\theta+C\cos^2\theta}-\bb{V_0}(\vp,\theta),
\end{array}
$$
where
$$
\begin{array}{l}
  Q=(B\cos^2\vp+A\sin^2\vp)C \cos^2\theta +AB \sin^2\theta, \\
  R=(A\cos^2\vp+B\sin^2\vp)C \sin^2\theta.
\end{array}
$$
Trajectories of the correspondent system are extremals of the functional
\begin{equation}\label{eq30}
  \int_{t_0}^{t_1}\mL dt
\end{equation}
in the class of curves lying on $S^2$ and satisfying the conditions $s(t_0)=s_0, s(t_1)=s_1$ ($s_0, s_1 \in S^2$). Introduce the following diffeomorphism $F:S^2\to E^2$
$$
x=\frac{1}{\sqrt{A}}\sin\theta\sin\vp, \quad y=\frac{1}{\sqrt{B}}\sin\theta\cos\vp, \quad z=\frac{1}{\sqrt{C}}\cos\theta.
$$
Extremals of the functional \eqref{eq30} under the map $F$ go to trajectories of the system on $E^2$ with the Lagrangian
$$
\tilde \mL = \mL\circ T F^{-1}=\frac{1}{2}ABC(\dot x^2+\dot y^2+\dot z^2)(A^2 x^2+B^2 y^2+C^2 z^2)^{-1}-V(x,y,z).
$$
In turn, these trajectories are extremals of the functional
\begin{equation}\label{eq31}
  \int_{t_0}^{t_1}\tilde \mL dt
\end{equation}
in the class of curves $u(t)\in E^2$ satisfying the conditions $u(t_0)=F(s_0), u(t_1)=F(s_1)$.

Let us fulfill the time change in the extremal problem of the functional
\begin{equation}\label{eq32}
  \int_{t_0}^{t_1} \left[A(u)T(\dot u)-U(u)\right]dt,
\end{equation}
where $T$ is a quadratic form of the components of the vector $\dot u$ with constant coefficients. The restricting relation has the form $\Phi(u)=0$, and the change is $dt=A(u)d\tau$. Let $\ou (t)$ be an extremal of the functional \eqref{eq32}. Then the Euler equations hold
\begin{equation}\label{eq33}
  \frac{d}{dt}\left[A(\ou)\frac{\prt T}{\prt \dot u}(\dot{\ou})\right]-\frac{\prt A}{\prt u}(\ou) T(\dot{\ou})+\frac{\prt U}{\prt u}(\ou)=\lambda(t)\frac{\prt \Phi}{\prt u}(\ou).
\end{equation}
The energy conservation law gives
$$
A(\ou)T(\dot \ou)+U(\ou)\equiv h.
$$
But
$$
T(\dot u)=A^{-2}(u)T(u'),
$$
where $u'=du/d\tau$, i.e.,
\begin{equation}\label{eq34}
  \frac{T(\ou')}{A(\ou)}\equiv h-U(\ou).
\end{equation}
Rewrite equation \eqref{eq33} in the following form
\begin{equation*}
\frac{1}{A(\ou)}  \frac{d}{d\tau}
\left[\frac{\prt T}{\prt u'}(\ou')\right]
-\frac{\prt A}{\prt u}(\ou) \frac{T(\ou')}{A^2(\ou)}+\frac{\prt U}{\prt u}(\ou)=\lambda(t)\frac{\prt \Phi}{\prt u}(\ou).
\end{equation*}
Substituting \eqref{eq34} we get
$$
\frac{d}{d\tau}
\left[\frac{\prt T}{\prt u'}(\ou')\right]
-\frac{\prt}{\prt u}\left[ T(\ou')-A(\ou)\left(U(\ou)-h\right)\right]=\tilde\lambda(t)\frac{\prt \Phi}{\prt u}(\ou),
$$
where $\tilde \lambda(t)=A(\ou(t))\lambda(t)$. This means that $\ou(t(\tau))$ is an extremal of the functional
$$
\int_{\tau_0}^{\tau_1}\left[T(u')-A(u)\left(U(u)-h\right)\right]d\tau
$$
with the energy constant $T(u')+A(u)\left(U(u)-h\right)\equiv 0$.

Applying this change of time to the functional \eqref{eq31} with $$
\displaystyle A(u)=\frac{ABC}{A^2x^2+B^2y^2+C^2z^2}, \quad T(\dot u)=\frac{1}{2}(\dot x^2+\dot y^2+\dot z^2), \quad U(u)=V(x,y,z),
$$
we obtain the first statement of the theorem.

To prove the second part, note that in the case of a free rigid body the corresponding mechanical system admits, as a symmetry group, the whole group $SO(3)$. Let $\ma =\mathfrak{so}(3)$ be the Lie algebra of $SO(3)$. The momentum integral is $j:M\to \ma^*$. Suppose that on a given trajectory $j=\mf$. Let us show that it is possible to choose a subgroup $G\subset SO(3)$ of rotations about some axis in space in such a way that on this trajectory the momentum integral of the mechanical system with symmetry $(M,K,V_0\equiv 0,G)$ equals zero. There exists $X_0\in \ma$ such that $\mf(X_0)=0$. For $G$, we choose a one-parameter subgroup of $SO(3)$ generated by $X_0$. Obviously, $G$ is the group of rotations about some axis fixed in space due to the uniqueness of the one-parameter subgroup with a given generator. Let $w_m\in TM$ and $j(w_m)=\mf$. Then (see \cite{Smale}, Proposition 4.7) $0=\mf(X_0)=K_m(w_m,\aq_m(X_0))$. The Lie algebra $\mG$ of $G$ is $\mG=\{\gamma X_0: \gamma\in \bR\} \subset \ma$ and for all $Y\in \mG$ we have $J(w_m)(Y)=K_m(w_m,\aq_m(Y))=\gamma K_m(w_m,\aq_m(X_0))=0$.

Thus, by choosing the appropriate axes in space and the corresponding symmetry group we can assure that on the investigated trajectories $J=0$. Applying the first statement of the theorem, we obtain the motion of a particle on the ellipsoid in the field with the potential $(-hABC)/(A^2x^2+B^2y^2+C^2z^2)$ with zero energy constant. Now, since the condition
$$
\frac{-hABC}{A^2x^2+B^2y^2+C^2z^2} <0
$$
holds everywhere, the proof of the theorem is completed by applying the Maupertuis principle
(see e.g. \cite{Arnold}).
\end{proof}

\begin{theorem}\label{theo5}
Consider the problem of the motion of a rigid body in an axially symmetric force field. For the zero momentum constant and any energy constant satisfying the condition $h > \max V_0$ there exist at least three motions such that each of them is periodic with respect to some frame of reference rotating with the constant angular velocity about the force field symmetry axis.
\end{theorem}
\begin{proof}
Applying to such a problem the statement of Theorem~\ref{theo4} we obtain that on the level $h>\max V)$ it is isomorphic to the geodesic flow on the ellipsoid $E^2$ in the metric
$d\Sigma_1=\sqrt{ABC(h-V)}(A^2 x^2+B^2 y^2+C^2 z^2)^{-1/2}d\sigma$, where $V:E^2\to \bR$ is defined by $V_0$ in Theorem~\ref{theo4} and $d\sigma$ is the metric on $E^2$ induced by the scalar product in $\bR^3$. According to the results of the work \cite{Lust}, this flow has at least three closed geodesics. These geodesics correspond to periodic solutions of the reduced system. Let $\vp=\vp(t)$, $\theta=\theta(t)$ be such a solution with a period $T$. Then by Remark~23, the corresponding solution of the initial problem is
\begin{equation}\label{eq35}
  \vp=\vp(t), \quad \theta=\theta(t), \quad \psi=\psi_0+\int_0^t f(\vp(\tau), \theta(\tau), \dot \vp(\tau), \dot \theta(\tau))d\tau,
\end{equation}
where $\dot \psi=f(\vp, \theta, \dot \vp, \dot \theta)$ is found from the equation $J=0$. Explicitly,
$$
\dot \psi =-\frac{(A-B)\dot \theta \sin\vp\cos\vp\sin\theta+C \dot \vp \cos\theta}{(A\sin^2\vp+B\cos^2\vp)\sin^2\theta+C\cos^2\theta}.
$$
Hence, $\dot \psi(t)$ is a periodic function with the period $T$. Then $\psi=\psi_0+\Lambda t+\Psi(t)$, where
$$
\Lambda=\frac{1}{T}\int_0^T\dot \psi(t)dt
$$
and $\Psi(t)$ is a $T$-periodic function. Consider a coordinate frame in the inertial space rotating with the angular velocity $\Lambda$ about the symmetry axis. With respect to this frame the trajectory \eqref{eq35} has the form $\vp=\vp(t)$, $\theta=\theta(t)$, $\psi=\psi_0+\Psi(t)$ and is $T$-periodic. The theorem is proved.
\end{proof}

\end{document}